\documentclass[11pt, reqno]{amsart}
\usepackage{amssymb,latexsym,amsmath,amsfonts,amsthm}

\numberwithin{equation}{section}

\usepackage[mathscr]{eucal}

\usepackage[T1]{fontenc}
\usepackage[pdftex,bookmarks=true]{hyperref}
\usepackage{graphicx}
\usepackage{enumerate}
\usepackage{graphicx}
\usepackage{enumerate}

\theoremstyle{plain}
\newtheorem{theorem}{Theorem}[section]

\newtheorem{lemma}[theorem]{Lemma}
\newtheorem{prop}[theorem]{Proposition}
\newtheorem{question}[theorem]{Question}

\theoremstyle{remark}
\newtheorem{rmk}[theorem]{Remark}

\theoremstyle{definition}
\newtheorem{dfn}[theorem]{Definition}

\newcommand{\cplx}{\mathbb{C}}

\newcommand{\cp}{\mathbb{CP}}
\newcommand{\del}{\partial}
\newcommand{\delbar}{\overline{\partial}}
\newcommand{\cdr}{\nabla}
\newcommand{\zbar}{\overline{z}}

\begin{document}
\title[Existence of a non-trivial laminations in $\cp^2 $]{On The Existence of non-trivial laminations in $\mathbb{CP}^2$}
\author{Divakaran Divakaran and Dheeraj Kulkarni}

\begin{abstract}
In this article, we show the existence of a non-trivial Riemann surface lamination embedded in $\mathbb{CP}^2$ by 
using Donaldson's construction of asymptotically holomorphic submanifolds. Further the lamination that we obtain has property that every leaf is a totally geodesic submanifold of $\cp^2 $ with respect to the Fubini-Study metric.
This may constitute a step in understanding the conjecture on the existence of minimal exceptional sets in $\cp^2 $. 
\end{abstract}

\maketitle

\section{Introduction}
A Riemann surface lamination is a compact Hausdorff space which is decomposed into disjoint 
union of Riemann surfaces, called leaves. Riemann surface laminations arise 
naturally in the context of complex dynamical systems generated by flows of 
complex vector fields. A Riemann surface lamination is called \emph{minimal} if all its leaves are dense.
It is well known that we can contruct minimal laminations in $\cp^n $ as flows of polynomial
vector fields, when $n$ is greater than or equal to $3$.  However, this
construction fails for $\cp^2 $ (see \cite{GE}). Motivated by this, Ghys asked the following question in the same article.
\begin{question}
Does there exist a minimal lamination that is holomorphically embedded in $\cp^2 $ and does not reduce to a compact Riemann surface?
\end{question} 

This is related to the question of existence of an exceptional minimal set for polynomial 
differential equations on $\cplx^2$, which is a formalized version of Hilbert's 
16$^{\text{th}} $ problem.  The absence of non-trivial minimal exceptional sets can be 
viewed as a generalization of the classical Poincar\'{e}-Bendixson theorem.
In \cite{GE}, Ghys notes that \emph{the question of existence of non-trivial minimal 
lamination embedded in $\cp^2$ is stronger than that of an "exceptional minimal set"}. 
In fact, there are no known examples of non-trivial Riemann surface laminations 
embedded in 
$\cp^2$.  Zakeri, in his article \cite{ZS}, states the following broader question 
whose special case is the earlier question asked by Ghys.
\begin{question}\label{zakeri_q}
 Does there exist a lamination that is holomorphically embedded in $\cp^2 $ and  not a compact Riemann surface?
\end{question}

We call a Riemann surface lamination, a {\em non-trivial} Riemann surface lamination if 
it is not a compact Riemann surface. In this article, we answer Question \ref{zakeri_q} 
in affirmative by proving the following theorem.

\begin{theorem}\label{main_thm}
There exists a non-trivial Riemann surface lamination $L$ embedded in 
$\cp^2 $. Moreover, each leaf of $L$ is a totally geodesic submanifold of $\cp^2 $ with respect to the Fubini-Study metric.
\end{theorem}

We construct a non-trivial Riemann surface lamination embedded in $\cp^2 $ by taking 
limits of regions of vanishing sets of asymptotically holomorphic sections constructed 
by Donaldson \cite{DS}. Donaldson used asymptotically holomorphic sections to show 
existence of symplectic submanifolds of any compact symplectic manifold. Further, 
Proposition 40 from \cite{DS} says that the currents
associated to the symplectic submanifolds converge to the fundamental form $
\frac{\omega_{FS}}{2\pi} $ of the Fubini-Study metric, which is supported everywhere on $\cp^2 $. 
Moreover, if there is a lamination
in the ``limit'', Bezout's Theorem suggests that there can be atmost one compact leaf in 
this lamination. Thus, one hopes to get a non-trivial Riemann surface lamination 
embedded in $\cp^2 $.

Continuing the above idea, we elaborate on how we can possibly obtain a holomorphic 
disk at 
a given point in $\cp^2$, which is the support of $\omega_{FS} $. Given any point $x $ 
in $\cp^2 $, we
can find a sequence of points $x_k $ belonging to these asymptotically holomorphic
submanifolds (zero sets of asymptotically holomorphic sections) which converges to $x
$. A uniform lower bound on the
injectivity radii of these asymptotically holomorphic charts near points $x_k $ will 
give a
\emph{smoothly} embedded disk at $x$. A version of Montel's Theorem (see Section \ref{approx_Montel}) for a family of asymptotically holomoprhic maps tells us that the
limiting disk is \emph{holomorphically} embedded. This will help us in constructing a 
chart for the
desired lamination near $x$.  A uniform lower bound on injectivity radii of these
asymptotically holomorphic submanifolds is given by the uniform upper bound on the 
second derivatives of asymptotically holomorphic sections.

We ``continue forward'' this disk to obtain a Riemann surface embedded in $\cp^2 
$.  This gives a recipe to construct a Riemann surface $S_x$ passing through 
arbitrary point $x \in \cp^2$. 
We, now, choose a point $x_0 \in \cp^2 $ and consider $S_{x_0} $.
If $S_{x_0} $ is compact, we choose a 
point $x_1 $ in the complement of $S_{x_0} $ and consider $S_{x_1} $. Further, our 
construction is such that $S_{x_0} \cap S_{x_1} = \emptyset $.  We observed 
(by Bezout's Theorem) earlier that $S_{x_1} $ can not be compact.

%


\begin{rmk}
The construction of non-trivial Riemann surface laminations in $
\cp^2 $ easily generalizes to the case of $\cp^n $ with $n \geq 3 $. The proof is 
essentially the same as the proof for the case of $\cp^2 $. This gives a 
new method of construction of non-trivial laminations in addition to the known 
method of considering flows of polynomial vector fields.
\end{rmk} 
\begin{rmk}
In this article, we make use of Donaldson's asymptotically holomorphic sections instead 
of ``locally concentrated'' holomorphic sections which he uses to give another proof 
of Kodaira embedding theorem.  We do so, because the second derivatives of the latter 
type of sections diverge to infinity (Corollary 33 in \cite{DS}), where as we have 
a uniform upper bound on the second derivatives of the sections of former type.       
\end{rmk}
\subsection*{Acknowledgements :}
The authors would like to thank Siddhartha Gadgil for posing a question about 
existence of Riemann surface lamination as a limit of asymptotically holomorphic 
submanifolds. The authors gratefully acknowledge the support from IISER Bhopal, 
India through the grant IISERB/INS/MATH/2016091.
The authors would like to thank E. Ghys for pointing out an error in the earlier version of this article.
\section{Preliminaries}
\subsection{Donaldson's construction of asymptotically holomorphic submanifolds}
Our argument to show the existence of a non-trivial embedded Riemann surface laminations is
based on Donaldson's construction of asymptotically holomorphic sections. Thus, in view of making our proof of existence theorem more accessible, we will give a brief exposition on
Donaldson's construction of asymptotically holomorphic sections. We employ the  notation in \cite{DS} as far as possible.

To 
recall the construction here, we start with the fundamental form $\omega_{FS} $ given by Fubini-
Study metric. We consider the tautological line bundle $\xi \rightarrow \cp^2 $ whose first Chern class 
is given by $\frac{\omega_{FS}}{2\pi} $. The $k$-th tensor power of $\xi$ is denoted by $\xi^{\otimes k}
$.    


The construction of asymptotically holomorphic sections proceeds in two stages, namely
the construction of locally supported asymptotically holomorphic sections followed by 
the construction of 
global section by taking suitable linear combination of locally supported sections.

First, we focus on the construction of locally supported asymptotically holomorphic 
sections. Let $\textbf{z} = (z_1, z_2) $ denote a point in $\cplx^2 $.
Donaldson makes use of Gaussian decay function $$f=e^{-\dfrac{|\textbf{z}|^2}{4}} $$ as 
a basic model for the construction of locally supported sections. On $\cplx^2 $, we 
consider the connection given by $$A =\frac{1}{4} \left( z_1d\zbar_1 -\zbar_1dz_1 + 
z_2d\zbar_2 - \zbar_2dz_2 \right) $$ then observe that $$i\, dA=\omega_{std}. $$
 We see that $$\delbar_A f = \delbar f + A^{0,1} f =0$$ Thus,
the Gaussian decay function is holomophic with respect to the coupled $\delbar $-operator $\delbar_A $. 
Further, Donaldson observes any non-integrable almost complex structure, when scaled 
sufficiently near a point, becomes close to being integrable. More precisely, the 
Nijenhius tensor with its higher order derivatives  can be made 
small in size under a suitable scaling transformation. 

Let $J_0$ denote the standard 
(integrable) complex structure on $\cplx^2 $. We consider the scaling map $\textbf{z} 
\mapsto k \textbf{z} $, for $k >0 $. We pull back the complex structure $J_0 $ by this 
scaling map and denote the pull-back by $\tilde{J}_k $.  The complexified 
cotangent bundle decomposes into complex linear and anti-linear parts 
with respect to both complex structures. Let $\Lambda^{1,0}_{J_0} $ and $\Lambda^{0,1}
_{J_0} $ denote the $J_0 $-linear and anti-linear parts respectively. The decomposition 
induced by $\tilde{J}_k $ can be tracked by a map $ \Lambda^{1,0}_{J_0} 
\rightarrow \Lambda^{0,1}_{J_0} $ which is linear at each point. We denote it by $\mu 
$. 

The derivative of the map $\mu$ is bounded by a factor of $k^{-\frac{1}{2}} $, where $k 
>0 $, denotes the scaling factor. One can achieve the effect of scaling, by a factor of $\sqrt{k} $ near a 
point in $\cplx^2 $ by taking $k$-th powers (under tensoring) of the trivial line bundle $\xi
$ over $\cplx^2 $, endowed with the connection $A$, i.e. by considering sections of the bundle $\xi^{\otimes k} $ over $\cplx^2 $ 
. The section $f$, which is holomorphic with respect to $J_0 $ and connection $A$, 
becomes approximately holomorphic with respect to $\tilde{J}_k $ and connection $A$. 
By putting all the above ingredients together, the local section is constructed by 
using 
$f$ and multiplying by a suitable cut-off function. We use Darboux charts $\chi: 
\mathbb{B}_r(0) \rightarrow \cp^2 $ on $\cp^2 $  
to push-forward these locally supported asymptotically holomorphic functions on $
\cplx^2 $. The Darboux charts can be extended to connection preserving bundle maps. 
Thus, we get locally supported sections of $k$-th tensor power of the tautological line 
bundle over $\cp^2 $. The crucial point here is that the Darboux charts are asymptotic 
isometries. Hence, all estimates on the derivatives hold true for locally supported 
sections on $ \cp^2 $ as well. Denote this locally supported section near point $p$ by 
$\sigma_p $.  

To obtain asymptotically holomorphic submanifolds as zero sets of asymptotically 
holomorphic 
sections, Donaldson takes complex linear combinations of locally supported 
asymptotically 
holomorphic sections in the following way. 

Let $\{B_i\} $ be a finite cover of $\cp^2 $ with each $B_i $ being the support of the 
section $\sigma_{p_i} $. We choose complex numbers $w_i $ with $\vert w_i \vert \leq 1 
$. Let $ \textbf{w}$ denote the tuple $(w_i) $. Then, we set $$s_{\textbf{w}} := \sum 
w_i \sigma_{p_i} $$
Often, we will denote the above section by $s_k $ also to emphasize the role of the twisting parameter $k$. The abosolute values $|\delbar s_k| $ and $| \cdr \delbar s_k| $ are of the size 
$O(k^{-1/2}) $. 
Donaldson uses estimated version of transversality to make sure that $| \del s_k| > 
\eta $ 
for some suitable $\eta >0 $. The estimated transversality is achieved step-by-step in 
the 
following way. Firstly, we partition the Darboux charts into $N(D)$ partions 
so that any two balls in the same partition are separated by distance $D (>0) $. The 
number of Darboux charts required to cover $\cp^2 $ grows with respect to parameter $k$ 
at least as fast as $k^4 $. Therefore, it is 
important to note that the number of partitions $N$ depends only on the separation 
distance $D$ and it is independent of the parameter $k$.  Now, start with the balls 
belonging to the first 
partition. We choose coefficients $w_i $ for locally supported sections on the balls in 
this 
partition so that some transversality is achieved on the union of balls in this  
partition. In the next step, we consider balls belonging to the second partion and 
adjust 
the coefficients for the local sections supported on these balls to achieve 
transversality 
on the new balls. However, we need to make sure that transversailty over the balls in 
the 
first partion is not completely lost. This is done by controlling the loss in the 
transversality by carefully choosing new coefficients corresponding to balls in the 
second 
partition. The process is repeated till we achieve transversality over the last 
partition 
without loosing transversality on the earlier partions.
Theorem 20 and in particular Proposition 23 from \cite{DS} states that the estimated 
transversality can be achieved for suitable choice of the separation parameter $D$ and 
for sufficiently large values of $k$. We reproduce the Proposition 23 from \cite{DS} 
below for later use in this article.

\begin{prop}\label{traserversality_extn}
Let $Q_p(\delta) = log(\delta^{-1})^p $. There are constants $\rho <1 $ and $p$ such 
that if section $S_{\underline{w}^{\alpha-1}}$ is $\eta_{\alpha -1} $-transverse over the union of balls belonging to all partitions starting from first till $\alpha -1 $ and if the twisting parameter $k$ and separation distance $D$ satisfy the following conditions
\begin{enumerate}
\item $k^{-1/2} < Q_p(\eta_{\alpha-1}) \eta_{\alpha-1} $,
\item $ \exp (-D^2/5) \leq Q_p(\eta_{\alpha-1})  $
\end{enumerate}
then there is another section $S_{\underline{w}^{\alpha}} $ such that it is $
\eta_{\alpha}$-transverse over union of balls belonging to all partitions $1,2,\ldots,
\alpha-1, \alpha$, where $\eta_{\alpha} = Q_p(\eta_{\alpha-1}) \eta_{\alpha-1} $.    
\end{prop}  

\subsection{Riemann surface laminations} A Riemann surface lamination is a compact space $M$ which locally looks like the product of a disk in the complex plane and a metric space.  We will make this notion precise.  An atlas for a Riemann surface lamination is given by:
\begin{itemize}
\item A cover by open sets $U_i$.
\item Homeomorphisms $\varphi_i:U_i \to \mathbb{D}\times T_i$ where $\mathbb{D}$ is a 
disk in $\cplx$ and $T_i$ is a topological space.
\item The transition maps satisfy the following property.  
$$\varphi_{ij}(z,t) := \varphi_j \circ \varphi_i^{-1}(z,t) = \left(\psi_{ij}(z,t), 
\lambda_{ij}(t)\right).$$ 
\end{itemize}

Two atlases are equivalent if their union is an atlas.  A Riemann surface lamination is 
a compact space $M$ equipped with an equivalence class of atlases.  The inverse images 
$\varphi_{ij}^{-1}\left(\mathbb{D}\times \{t\}\right)$ are called plaques.  Consider 
the relation $p \sim q$ if $p$ and $q$ lie on the same plaque.  This relation is 
reflexive and symmetric.  Consider the transitive closure of this relation and call it 
$\sim$.  Equivalence classes of this equivalence relation are called leaves.  A 
lamination is called minimal if all its leaves are dense.  

\section{Step 1: Second derivative estimates for $s_k$}

In this section, we establish bounds on the second derivative of Donaldson's asymptotically 
holomorphic sections.

%

We note the following facts which are easy to verify 
\begin{equation}\label{delbar_part}
\overline{\del}_A \left( e^{- \dfrac{\vert{\bf z} \vert^2}{4}} \right) = 0 
\end{equation}
\begin{align}\label{del_part}
\del_A \left( e^{- \dfrac{\vert{\bf z} \vert^2}{4}} \right) &= (\del + A^{1,0}) e^{- \dfrac{\vert{\bf z} \vert^2}{4}} 
\nonumber \\
& = -\frac{1}{2} \left( \overline{z}_1 d z_1 + \overline{z}_2 d z_2 \right) e^{- \dfrac{\vert{\bf z} \vert^2}{4}}
\end{align}

Combining the equations \ref{delbar_part} and \ref{del_part}, we get
\begin{equation}\label{first_deri}
\cdr_A \left( e^{- \dfrac{\vert{\bf z} \vert^2}{4}} \right) = (d + A)  e^{- \dfrac{\vert{\bf z} \vert^2}{4}} = -\frac{1}{2} \left( \overline{z}_1 d z_1 + \overline{z}_2 d z_2 \right) e^{- \dfrac{\vert{\bf z} \vert^2}{4}}
\end{equation}

Therefore, we observe the following estimate on the first covariant derivative.
\begin{equation}\label{fde}
\left\vert\cdr_A \left( e^{- \dfrac{\vert{\bf z} \vert^2}{4}} \right) \right\vert  \leq \frac{1}{2} \vert {\bf z}\vert e^{- \dfrac{\vert{\bf z} \vert^2}{4}}
\end{equation}

Now, we take the second covariant derivative of the term in the rightmost side of Equation 
\ref{first_deri}. We get
\begin{align}\label{second_deri}
\cdr_A \left( \cdr_A \left( e^{- \dfrac{\vert{\bf z} \vert^2}{4}} \right) \right) &= \cdr_A -\frac{1}{2} \left( \overline{z}_1 d z_1 + \overline{z}_2 d z_2 \right) e^{- \dfrac{\vert{\bf z} \vert^2}{4}} \nonumber \\
&= \underbrace{\cdr_A^{1,0}\left( -\frac{1}{2} \left( \overline{z}_1 d z_1 + \overline{z}_2 d z_2 \right) e^{- \dfrac{\vert{\bf z} \vert^2}{4}}\right) }_\text{(1)} +  \underbrace{\cdr_A^{0,1}\left(-\frac{1}{2} \left( \overline{z}_1 d z_1 + \overline{z}_2 d z_2 \right) e^{- \dfrac{\vert{\bf z} \vert^2}{4}} \right) }_\text{(2)}.
\end{align}

The term $(1) $ in the above expression gives $0$ after a straightforward computation and the expression in $(2) $ of Equation \ref{second_deri} can be simplified to
\[ \frac{1}{2} \left(dz_1\wedge d \overline{z}_1 + dz_2\wedge d \overline{z}_2\right) e^{- \dfrac{\vert{\bf z} \vert^2}{4}} \]

We note that both the terms in Equation \ref{second_deri} are bounded in absolute value 
by $Ce^{- \dfrac{\vert{\bf z} \vert^2}{4}} $, where $C$ is a suitable constant.

We now need to transfer the estimates onto the image of any Darboux chart. As observed 
in \cite{DS} the Darboux charts $\chi \circ \delta_{k^{-1/2}} $ are asymptotic 
isometries, thus the estimates on the first and second order derivatives will change 
atmost by a factor of $Ck^{-1/2}|\textbf{z}|^3 $ (see Proposition 11 in \cite{DS}). We 
recall some notation from the same article by Donaldson.
Let $\sigma_p $ denote a section obtained by push-forward of locally supported 
asymptotically holomorphic section constructed as above. Let $d$ denote the distance 
induced by the Fubini-Study metric. The scaled distance $k^{1/2}d$ is denoted by $d_k 
$. Define the symbol $e_k(p,q) $ to be $e^{-d_k^2(p,q)/5} $ if $d_k(p,q) \leq k^{1/4} $ 
and is it $0 $ for $d_k(p,q) > k^{1/4} $. Then, the section $\sigma_p $, for each $p 
\in \cp^2 $, the following estimates hold: 
\begin{enumerate}
\item $\vert \sigma_p (q) \vert  \leq e_k(p,q) $,
\item $\vert \cdr \sigma_p (q) \vert \leq C\left(1+d_k(p,q)\right)e_k(p,q)  $,
\item $\vert \delbar \sigma_p(q) \vert \leq Ck^{-1/2}d_k^2(p,q)e_k(p,q) $,
\item $\vert \cdr \delbar \sigma_p(q)  \vert \leq  Ck^{-1/2} \left(d_k(p,q)+d_k^3(p,q)  \right)e_k(p,q)$.
\end{enumerate}
where $\cdr $ denotes the Levi-Civita connection on $\cp^2$ of the Fubiny-Study 
metric and $C$ is a constant independent of $k$.

From the above inequalties it follows that 
\begin{equation}
\left\vert \cdr  \left( \cdr \sigma_p \right) \vert_q \right\vert < C \left( d_k(p,q) + d_k(p,q)^3 
\right) e_k(p,q)
\end{equation}   
Therefore, we see that the second covariant derivative is bounded by some $m>0$ as 
given below
\begin{equation}\label{loc_sec_deri_est}
|\cdr \left( \cdr \sigma_p \right) | < m
\end{equation}
To get a global estimates on $s_{\textbf{w}} $ and its derivatives,
%
%
we note below the estimates given in Lemma 14 in \cite{DS} for the later use in this article.
\begin{lemma}
\label{unibounds}
For any choice of coefficients $\textbf{w}= (w_i) $ with $\vert w_i \vert \leq 1 $, the section $s_{\textbf{w}} $ satisfies the following estimates everywhere on $\cp^2 $
\begin{align*}
\vert s_{\textbf{w}} \vert &\leq C ,\\
\vert \delbar s_{\textbf{w}}  \vert &\leq Ck^{-1/2},\\
\vert \cdr \delbar s_{\textbf{w}}\vert &\leq Ck^{-1/2}. 
\end{align*}
\end{lemma}

The following lemma builds on the Lemma 12 in \cite{DS}. It plays central role in establishing the fact that lamination we obtain is totally geodesic.

\begin{lemma} \label{glo_sec_der_est} 
At each point $p \in \cp^2  $, the following holds,
\begin{itemize}
\item  $\left\vert \left(\cdr  \cdr s_{k } \right)\vert_p \right\vert \longrightarrow 0
$ 
\item $\left\vert d \left(\vert \cdr s_{k}  \vert \right) \vert_p \right\vert \longrightarrow 0$
\end{itemize}
as $k \rightarrow \infty $. However, the convergence is not uniform. 
\end{lemma}
\begin{proof}
In the proof of Lemma 12 in \cite{DS}, Donaldson reduces the argument to 
that of Euclidean case. Donaldson chooses a cover of $\cplx^n$  with balls having their centres at the lattice points of some suitable lattice in $
\cplx^n $. Then the sum \ref{lattice_sum} over the chosen lattice is observed to be uniformly bounded.
Our argument is the same as in Lemma 12 in \cite{DS}. We combine it with the 
observation that for each $a, r >0$ and $\omega \in \cplx^2 $ the infinite 
sum over a lattice $\Lambda $ in $\cplx^2 $.
\begin{align}\label{lattice_sum}
\sum_{\mu \in \Lambda} \left\vert \mu -\omega \right\vert^r e^{-a \left\vert 
\mu -\omega \right\vert^2}
\end{align}
converges to $0$ pointwise (but not uniformly) with respect to $\omega $, as 
$a \rightarrow \infty$. 
\end{proof}


\section{Step 2: Construction of approximately holomorphic disks around points in $s_k^{-1}(0)$}

In this section, we show the existence of embedded disks in the vanishing sets of 
approximately holomorphic sections $s_k $ such that the radii of these disks are 
bounded below by some positive constant $r$ which is independent of $k$ and point 
$p$. 
Moreover, we show that there is ``approximately holomorphic'' embeddings of the disk 
of radius $r$. 

Let $s_k$ be a section of the line bundle $\xi^{\otimes k} \to \mathbb{CP}^2$ such that 
$s_k^{-1}(0)$ 
is an ``approximately'' holomorphic  submanifold of $\mathbb{CP}^2$.  On each affine chart $A_j$ for $
\mathbb{CP}^2$, we think of $s_k$ as a function $s_k:\mathbb{C}^2\to \mathbb{C}$.  We 
further break $s_k$ into real and imaginary parts, $s_k = u_k + iv_k$. Notice that $
\frac{\nabla u}{\vert \nabla u \vert} = N_1$ and $\frac{\nabla v}{\vert \nabla v \vert} 
= N_2$ are two unit normal vector fields to $s_k^{-1}(0)$. In this section, we 
will represent the curvature of $s_k^{-1}(0)$ in terms of the the curvature of $
\mathbb{CP}^2$ and Weingarten 
maps.  Thus an upper bound on the norm of Weingarten maps would give an upper 
bound on the curvature of $s_k^{-1}(0)$.  We will prove the upper bound on 
Weingarten maps, by giving 
bounds on the derivatives of $\frac{\nabla u_k}{\vert \nabla u_k \vert}$, $
\frac{\nabla 
v_k} {\vert \nabla v_k \vert}$ , in Section \ref{Weinga}. The upper bound on curvature 
will in turn give a lower bound on radii of 
embedded disks in $s_k^{-1}(0)$ by the celebrated result of Klingenberg. 

Recall that $\nabla$ denotes the Levi-Civita connection induced by Fubini-Study metric on $
\mathbb{CP}^2$. Let $ \overline{\nabla}$ denote the Levi-Civita connection on 
$s_k^{-1}
(0) = V$ induced by the Fubini-Study metric on $\mathbb{CP}^2$.  We define the two 
Weingarten maps  $w_{i,p}: T_p V \times T_p V 
\to \mathbb{R}$ as follows
\begin{align}
w_{1,p}(u,v) &= \langle \nabla_u N_1 - \langle \nabla_u N_1,N_2 
\rangle N_2, v \rangle\\
w_{2,p}(u,v) &= \langle \nabla_u N_2 - \langle \nabla_u N_2,N_1 \rangle N_1, v 
\rangle.
\end{align}

\noindent For vector fields $X, Y, Z$ on $\mathbb{CP}^2$ (or an affine chart), the 
curvature tensor is given by
$$R_{\nabla}(X,Y,Z) = \nabla_X \nabla_Y Z - \nabla_Y \nabla_X Z - \nabla_{[X,Y]}Z.$$
As we are interested in computing the curvature of $V$, we would like to compute values 
of 
$\langle R_{\nabla}(X,Y,X),Y \rangle$ and $\langle R_{\overline{\nabla}}(X,Y,X),Y 
\rangle$ when $X,Y$ are tangent to $V$.  As $Y$ is orthogonal to $N_1$ and $N_2$, we 
will ignore the components of $R_{\nabla}(X,Y,X)$ along $N_1$ and $N_2$.   

As $T_p \mathbb{CP}^2$ is spanned by $T_p V$, $N_1(p)$ and $N_2(p)$, we can write, 
\begin{align}\label{rel_connection}
\nabla_Y X = \overline{\nabla}_Y X + \langle \nabla_Y X,N_1\rangle N_1 + \langle 
\nabla_Y X,N_2\rangle N_2. 
\end{align}
So,
\[
\nabla_X \nabla_Y X = \nabla_X (\overline{\nabla}_Y X) + \nabla_X (\langle \nabla_Y X,N_1\rangle N_1) + \nabla_X (\langle \nabla_Y X,N_2\rangle N_2).
\]
We will compute each term in the right hand side separately.  
Note that, as $\langle X, N_i\rangle = 0$,  
\[
\langle \nabla_Y X,N_i\rangle = -\langle \nabla_Y N_i,X\rangle
\]
Thus, we rewrite (\ref{rel_connection}) as
\begin{align}\label{rel_connection_1}
\nabla_Y X = \overline{\nabla}_Y X -\langle \nabla_Y N_1,X\rangle N_1 -\langle \nabla_Y N_2,X\rangle N_2. 
\end{align}

Now, we take futher covariant derivative and obtain the following
\[
\nabla_X (\overline{\nabla}_Y X) = \overline{\nabla}_X \overline{\nabla}_Y X - \left
\langle \overline{\nabla}_X N_1, \overline{\nabla}_Y X \right\rangle N_1 - \left\langle 
\overline{\nabla}_X N_2, \overline{\nabla}_Y X \right\rangle N_2 
\]
and
\[
\nabla_X (\langle \nabla_Y X,N_i\rangle N_i) = -\left\langle \nabla_Y N_i, X \right\rangle \nabla_X N_i + \text{components along $N_1$ and $N_2$}.
\]

For simplicity, we will call the components along $N_1$ and $N_2$, normal components. So, we have,
\[
\nabla_X \nabla_Y X = \overline{\nabla}_X \overline{\nabla}_Y X -\left\langle \nabla_Y N_1, X \right\rangle \nabla_X N_1 - \left\langle \nabla_Y N_2, X \right\rangle \nabla_X N_2 + \text{normal components}.
\]
Similarly,
\[
\nabla_Y \nabla_X X = \overline{\nabla}_Y \overline{\nabla}_X X -\left\langle \nabla_X N_1, X \right\rangle \nabla_Y N_1 - \left\langle \nabla_X N_2, X \right\rangle \nabla_Y N_2 + \text{normal components}.
\]
And finally,
\[
\nabla_{[X,Y]} X = \overline{\nabla}_{[X,Y]} X + \text{normal components}.
\]
Combining the above expressions, we have,
\begin{align*}
R_{\nabla}(X,Y,X) &= \overline{\nabla}_X \overline{\nabla}_Y X -\left\langle \nabla_Y N_1, X \right\rangle \nabla_X N_1 - \left\langle \nabla_Y N_2, X \right\rangle \nabla_X N_2 - \overline{\nabla}_Y \overline{\nabla}_X X +\left\langle \nabla_X N_1, X \right\rangle \nabla_Y N_1\\
& \ \ \ \ \ \ \ \ + \left\langle \nabla_X N_2, X \right\rangle \nabla_Y N_2 - \overline{\nabla}_{[X,Y]} X + \text{normal components}\\
&= R_{\overline{\nabla}}(X,Y,X) -\left\langle \nabla_Y N_1, X \right\rangle \nabla_X N_1 - \left\langle \nabla_Y N_2, X \right\rangle \nabla_X N_2 +\left\langle \nabla_X N_1, X \right\rangle \nabla_Y N_1\\
& \ \ \ \ \ \ \ \ + \left\langle \nabla_X N_2, X \right\rangle \nabla_Y N_2 + \text{normal components}
\end{align*}
Now, taking $X,Y$ to be orthonormal vector fields spanning $TV$, we get the curvature of $V$ in $\mathbb{CP}^2$ to be
\begin{align*}
K^{\nabla}_V &= \left\langle R_{\nabla}(X,Y,X),Y \right\rangle\\
&= \left\langle R_{\overline{\nabla}}(X,Y,X), Y \right\rangle -\left\langle \nabla_Y 
N_1, X \right\rangle \left\langle \nabla_X N_1, Y \right\rangle - \left\langle \nabla_Y 
N_2, X \right\rangle \left\langle \nabla_X N_2, Y \right\rangle +\left\langle \nabla_X 
N_1, X \right\rangle \left\langle \nabla_Y N_1,Y \right\rangle\\
& \ \ \ \ \ \  + \left\langle \nabla_X N_2, X \right\rangle \left\langle \nabla_Y N_2, 
Y\right\rangle\\
&= K^{\overline{\nabla}}_V + w_1(X,X)w_1(Y,Y) + w_2(X,X)w_2(Y,Y) - w_1(X,Y)w_1(Y,X) - 
w_2(X,Y)w_2(Y,X).
\end{align*}
The bounds on the curvature of $\cp^2$  endowed with Fubini-Study metric($1/4 \leq 
K \leq 1 $) and the Weingarten maps (see  Lemma \ref{weinga} ) will give us the following bound on the curvature of 
$V$ 
\[
 \left\vert K^{\overline{\nabla}}_V \right\vert \leq 1+4 \beta^2
\]
where $\beta >0 $ is some contanst independent of $k$.
Now, we apply Klingenberg's theorem along with the upper bound on the curvature 
obtained as 
above to give a lower bound on the injectivity radius of $V$, say $r$ that depends 
only on $\beta $. Note that this bound is independent of $k$. 

Now, we define what we mean by an approximately holomorphic embedding of a disk.
\begin{dfn}
Let  $\epsilon >0 $ and $D \subset \cplx $ be a domain. We say that a smooth embedding
$\varphi : D \rightarrow \cplx^2 $ is \emph{$\epsilon$-approximately holomorphic} if 
the following holds.
\[ \left\vert \frac{\del \varphi}{\del \overline{z} } \right\vert < \epsilon.\] 
\end{dfn}
We will call an $\epsilon $-approximately holomophic embedding of a disk $D_r $ as  an \emph{approximately holomorphic} disk in $\cplx^2 $. We construct an approximately holomorphic disk $\varphi: D_r \rightarrow V $ as follows. For a point $p \in V $, 
let $L $ denote a 1-dimensional complex subspace of $T_p \cp^2 $ which is close to $T_p 
V \subset T_p \cp^2 
$. Such an $L$ exists because $T_pV $ is approximately complex vector subspace of $T_p
\cp^2 \left(=\cplx^2 \right) $. Let $\pi $ denote the projection of $L $ onto $T_pV $. 
We note that the antilinear part of $\pi $ satisfies the bound given below
$$\Vert \pi^{0,1} \Vert < Ck^{-1/2}.$$
Let $exp_p $ denote the exponential map 
$T_pV \rightarrow V $. Then consider the map $\varphi = \exp_p \circ \pi : D_r 
\rightarrow \cplx^2 $, where $D_r \subset L$ is a disk centred at origin of radius $r$. We observe that the following lemma holds

\begin{lemma}\label{approx_holo_disk}
The disk $\varphi $ is $\epsilon $-approximately holomorphic, where $\epsilon = C'k^{-1/2} $.
\end{lemma}    
\begin{proof}
Notice that the derivative of $\exp_p $ is bounded by some constant $\tau $. The 
tangent vector $\frac{\del}{\del\overline{z}} $ to $D_r $ gets mapped to $\frac{\del 
\varphi}{\del\overline{z}} $ whose norm is bounded above by $\tau C k^{-1/2} $. We  
take $C' = \tau C $. 
\end{proof}

\section{Step 3: Estimates on the Weingarten Maps}\label{Weinga}
First, we consider the affine charts $A_i $, for $i =0,1,2 $, for $\cp^2 $. On each 
chart, the bundle $\xi^{\otimes k} $ is trivial. Thus we can express the section $s_k $ 
of $\xi^{\otimes k} $ 
as a function $f: \cplx^2 \rightarrow \cplx $. Further, we write $f_k = u_k + i v_k $, 
where $u_k$ and $v_k$ are real valued functions. We  pull-back of the Fubini-Study metric on $A_j $ and denote the pull-back of the connection by $\cdr $. Then, the complex gradient  $\cdr f_k$ can be written as $\cdr u_k + i 
\cdr v_k $. Note that we have the following estimate as $s_k $ is asymptotically 
holomorphic
\begin{equation}\label{real_im_der_same}
|i\cdr u_k - \cdr v_k | < Ck^{-\frac{1}{2}}
\end{equation}

\begin{lemma}\label{eta_lower_bound}
There exists $\eta >0$, independent of $k$, such that $|\cdr u_k| > \eta $ and $| \cdr 
v_k| > \eta $.
\end{lemma}
\begin{proof}
In the proof of estimated transversality (Proposition 23 in \cite{DS} or see 
Proposition \ref{traserversality_extn}), the separation parameter $D$ is 
chosen in the end to retain some transversality. The separation parameter $D$ was 
independent of $k$. This means for all sufficiently large $k$, $|\cdr s_k| > c $, 
for some $c>0 $ independent of $k$. Inequality \ref{real_im_der_same} implies that $|
\cdr u_k| $ and $|\cdr v_k | $ are asymptotically of same size. Combining it with 
Donaldson's estimated transversality, we conclude that both derivatives can not become 
arbitrarily small when $k$ is sufficiently large.

\end{proof}      

\begin{lemma}\label{weinga}
Let $N_1^{\left( k \right)} = \frac{\cdr u_k}{\vert \cdr u_k \vert} $ and $N_2^{\left( k \right)} = \frac{\cdr v_k}{\vert \cdr v_k \vert} $. 
Then there is a constant $\beta >0 $ which is independent of $k$ and point $p \in 
\cp^2$ such that the following inequalities hold 
\begin{align}
\vert \cdr N_1^{\left( k \right)} \vert &\leq \beta,  \\
\vert \cdr N_2^{\left( k \right)} \vert &\leq \beta.
\end{align} 
\end{lemma}
\begin{proof}
The proof follows by straightforward computation of the derivative and applying the 
estimates obtained in the previous lemma. We do the computation for the vector field $ 
N_1^{\left( k \right)}$ explicitly. The estimate on $\vert\cdr N_2^{\left( k \right)} \vert $ follows similarly.

\begin{align*}
\cdr \left(\frac{\cdr u_k}{\vert\cdr u_k \vert } \right) = \underbrace{\frac{1}{\vert
\cdr u_k \vert} \cdr \left( \cdr u_k \right) }_\text{(A)} + \underbrace{\cdr u_k 
\otimes d\left(\frac{1}{\vert \cdr u_k \vert } \right) }_\text{(B)}
\end{align*}
In the term $(A)$, by applying previous lemma we see that 
\begin{align}\label{denominator_bdd}
\frac{1}{\vert\cdr u_k \vert} \leq \frac{1}{\eta}
\end{align}
Now. observe that
\[ \vert \cdr \left( \cdr u_k \right) \vert \leq \vert \cdr \left( \cdr s_k \right) 
\vert\]

and by Inequality \ref{glo_sec_der_est} we have,
\[\vert \cdr \left( \cdr u_k \right) \vert \leq \vert \cdr \left( \cdr s_k \right) 
\vert < M \]
Combining the above we get 
\begin{align}\label{termA}
\left\vert \frac{1}{\vert \cdr u_k \vert} \cdr \left( \cdr u_k \right) \right\vert < \frac{M}{\eta} 
\end{align}

\noindent To obtain the estimates in the term $(B) $, we note that 
\[ \vert \cdr u_k \vert \leq \vert \cdr s_k \vert < M' \]
and 
\[ \left\vert d\left(\frac{1}{\vert \cdr u_k \vert } \right) \right\vert = \left\vert -\frac{1}{\vert \cdr u_k \vert^2} d \left( \vert \cdr u_k \vert \right) \right\vert \leq 
\frac{C}{\eta^2} \]
The bound on $\vert d \left( \vert \cdr u_k \vert \right)  \vert $ follows from 
\ref{first_deri} as the term $e^{-\frac{|\sqrt{k}\textbf{z}|}{4}} $ survives after differentiating.  Combining the above two inequalties we observe that
\begin{align}\label{termB}
\left\vert \cdr u_k \otimes d\left(\frac{1}{\vert \cdr u_k \vert } \right) \right\vert \leq \frac{M'C}{\eta^2}
\end{align}
The inequalities \ref{termA} and \ref{termB} together imply that
\[ \vert \cdr N_1 \vert \leq \frac{M}{\eta} +\frac{M'C}{\eta^2} \]
where all the constants are independent of $k$.

\end{proof}

We will now strengthen the above lemma by showing that the derivatives of the normals go to zero pointwise as $k $ goes to infinity.

\begin{lemma}\label{pointwise_zero}
For each point $p \in \cp^2 $, as $k \longrightarrow \infty $, we have the following
\begin{align*}
\vert \cdr N_1^{\left( k \right)} \vert &\longrightarrow 0   \\
\vert \cdr N_2^{\left( k \right)} \vert &\longrightarrow 0 .
\end{align*}
\end{lemma} 
\begin{proof}

By Lemma \ref{glo_sec_der_est} we have,
$\vert \cdr \left( \cdr u_k \right) \vert \longrightarrow 0 $ as $k \rightarrow \infty $ pointwise.
Recall that from \ref{denominator_bdd} we have $\vert \cdr u_k \vert ^{-1} < \eta $ for some constant $\eta >0 $ independent of $k$ as given in Lemma \ref{eta_lower_bound}. Therefore, we see that
\begin{align*}
\left\vert \frac{1}{\vert \cdr u_k \vert} \cdr \left( \cdr u_k \right) \right\vert   \longrightarrow 0 
\end{align*}
as $k \rightarrow \infty $.\\ 
\noindent 
Recall that 
\[ \left\vert d\left(\frac{1}{\vert \cdr u_k \vert } \right) \right\vert = \left\vert -\frac{1}{\vert \cdr u_k \vert^2} d \left( \vert \cdr u_k \vert \right) \right\vert \]
Further $ \vert \cdr u_k \vert ^{-2} < \eta^2$ from Lemma\ref{eta_lower_bound}.
Observe that the real part of $ d\vert \cdr s_k \vert$ gives $ d\vert u_k 
\vert $. By Lemma \ref{glo_sec_der_est} 
$ d\vert \cdr s_k \vert \longrightarrow 0$ as $k \longrightarrow 0 $.
\begin{align*}
\left\vert \cdr u_k \otimes d\left(\frac{1}{\vert \cdr u_k \vert } \right) \right\vert \longrightarrow 0
\end{align*}
as $k \rightarrow \infty $ for each point. 
\end{proof}

\section{Step 4: Approximate version of Montel's theorem}\label{approx_Montel}
Let $D_r$ be a disk of radius $r$ in $\cplx$ centred at the origin. Further, let $
\varphi_k: D_r \to \cplx^2$ be approximately holomorphic disks contructed earlier.  Let 
$p_i: \mathbb{C}^2 \to \mathbb{C}$ be the projection map to the $i$-th coordinate, $i=1,2 $. Note 
that, $\psi^i_k = p_i\circ \varphi_k: D_r \to \mathbb{C}$ are approximately 
holomorphic.  

\begin{lemma}
Some subsequence of $\psi^i_k$ converges to a function $\psi^i: D_r \to \mathbb{C}$.
\end{lemma}

\begin{proof}
We will prove this by showing that the family of functions $\psi^i_k$ is uniformly bounded and equi-continuous.\\
\noindent \textbf{Uniform boundedness:}\\
 By Cauchy-Pompeiu formula, given a disk $D \subset D_r$, 
\[
\psi^i_k(\zeta) = \frac{1}{2\pi \iota} \int_{\partial D} \frac{\psi^i_k(z) dz}{z - \zeta} - \frac{1}{\pi} \iint_D \frac{\partial \psi^i_k}{\partial \overline{z}} \frac{dx\wedge dy}{z - \zeta}.
\]  
So,
\begin{align*}
\vert \psi^i_k(\zeta) \vert \leq \frac{1}{2\pi} \int_{\partial D} \frac{\vert \psi^i_k(z) \vert \vert dz \vert}{\vert z - \zeta \vert} + \frac{1}{\pi} \iint_D \left\vert \frac{\partial \psi^i_k}{\partial \overline{z}} \right\vert \frac{dx\wedge dy}{\vert z - \zeta \vert}.  
\end{align*} 
By Lemma \ref{approx_holo_disk}, $\vert \psi^i_k(z) \vert < M_1$ and $\left\vert \frac{\partial \psi^i_k}{\partial \overline{z}} \right\vert < M_2$ for constants $M_1, M_2$ independent of the function and the point $z$.  Thus,
\begin{align*}
\vert \psi^i_k(\zeta) \vert \leq \frac{M_1}{2\pi} \int_{\partial D} \frac{\vert dz\vert}{\vert z - \zeta \vert} + \frac{M_2}{\pi} \iint_D  \frac{dx\wedge dy}{\vert z - \zeta \vert}
\end{align*}
Let $g: \partial D \to \mathbb{R}$ given as $g(z) = \vert z- \zeta \vert$.  This is 
clearly continuous and $\partial D$ is compact, so $g$ attains its minimum, say $c$.  
Thus, if $D$ is a disk of radius $R$,
\[
\int_{\partial D} \frac{\vert dz \vert}{\vert z - \zeta \vert} \leq \int_{\partial D} \frac{\vert dz\vert }{c} = \frac{2\pi R}{c}.
\]
On the other hand, 
\[
\iint_D  \frac{dx\wedge dy}{\vert z - \zeta \vert} = \int_0^{R} \int_0^{2\pi}  \frac{r dr d\theta}{r} = \int_0^{R} \int_0^{2\pi} dr d\theta = 2\pi R.
\]
       The integrability of the integrals gives us uniform boundedness. \\ 
\noindent \textbf{Equi-continuity:}\\
\begin{align*}
\vert \psi^i_k(\zeta_1) - \psi^i_k(\zeta_2) \vert &\leq \left\vert \frac{1}{2\pi \iota} 
\int_{\partial D} \frac{\psi^i_k(z) dz}{z - \zeta_1} - \frac{1}{2\pi \iota} 
\int_{\partial D} + \frac{\psi^i_k(z) dz}{z - \zeta_2}\right\vert \\
& \ \ \ \ + \left\vert \frac{1}{\pi} \iint_D \frac{\partial \psi^i_k}{\partial 
\overline{z}} \frac{dx\wedge dy}{z - \zeta_2} - \frac{1}{\pi} \iint_D \frac{\partial 
\psi^i_k}{\partial \overline{z}} \frac{dx\wedge dy}{z - \zeta_1} \right\vert  
\end{align*}
We know from the proof of Montel's theorem that $\left\vert \frac{1}{2\pi \iota} 
\int_{\partial D} \frac{\psi^i_k(z) dz}{z - \zeta_1} - \frac{1}{2\pi \iota} 
\int_{\partial D} + \frac{\psi^i_k(z) dz}{z - \zeta_2}\right\vert < \varepsilon$ if $
\vert \zeta_1 - \zeta_2 \vert <\delta$, where $\delta $ is independent of functions $\psi^i_k $. On the other hand,
\begin{align*}
\left\vert \frac{1}{\pi} \iint_D \frac{\partial \psi^i_k}{\partial \overline{z}} 
\frac{dx\wedge dy}{z - \zeta_2} - \frac{1}{\pi} \iint_D \frac{\partial \psi^i_k}
{\partial \overline{z}} \frac{dx\wedge dy}{z - \zeta_1} \right\vert &\leq \left\vert 
\frac{1}{\pi} \iint_D  \frac{\partial \psi^i_k}{\partial \overline{z}} \frac{\zeta_1 - 
\zeta_2}{(z - \zeta_1)(z- \zeta_2)} dx\wedge dy \right\vert\\
&\leq \vert \zeta_1 - \zeta_2 \vert \left\vert \frac{1}{\pi} \iint_D  \frac{\partial 
\psi^i_k}{\partial \overline{z}} \frac{1}{(z - \zeta_1)(z- \zeta_2)} dx\wedge dy\right
\vert
\end{align*}
Thus, we have equi-continuity if we have a uniform bound on $\left\vert \iint_D  
\frac{\partial \psi^i_k}{\partial \overline{z}} \frac{1}{(z - \zeta_1)(z- \zeta_2)} dx
\wedge dy\right\vert$.  We will first prove $ \frac{1}{(z - \zeta_1)(z- \zeta_2)}$ is integrable and later use Stoke's 
theorem to obtain a bound.  From the estimates on second derivatives, we know that $
\left\vert \frac{\partial \psi^i_k}{\partial \overline{z}} \right\vert < M_2$ by Lemma \ref{approx_holo_disk}.  So,
\begin{align*}
\left\vert \iint_D  \frac{\partial \psi^i_k}{\partial \overline{z}} \frac{1}{(z - 
\zeta_1)(z- \zeta_2)} dx\wedge dy \right\vert &\leq \iint_D \left\vert\frac{\partial 
\psi^i_k}{\partial \overline{z}}\right\vert\left\vert \frac{1}{(z - \zeta_1)(z- 
\zeta_2)} \right\vert dx\wedge dy\\
&\leq M_2 \iint_D \left\vert \frac{1}{(z - \zeta_1)(z- \zeta_2)} \right\vert dx\wedge 
dy\\
\end{align*} 

We will further split the integral on the right hand side of the above inequality as follows.

\begin{align*}
\iint_D \left\vert \frac{1}{(z-\zeta_i)(z-\zeta_2)} \right\vert dx\wedge dy &= \iint_{B_{\rho}(\zeta_1)} \left\vert \frac{1}{(z-\zeta_1)(z-\zeta_2)} \right\vert dx\wedge dy\\
& \  + \iint_{B_{\rho}(\zeta_2)} \left\vert \frac{1}{(z-\zeta_1)(z-\zeta_2)} \right\vert dx\wedge dy\\
& \ + \iint_{D\setminus (B_{\rho}(\zeta_1)\cup B_{\rho}(\zeta_2)} \left\vert \frac{1}{(z-\zeta_1)(z-\zeta_2)} \right\vert dx\wedge dy\\
\end{align*}
Choose $\rho$ such that $0<\rho< \frac{d(\zeta_1 - \zeta_2)}{2}.$  Then, notice that,
\begin{align}
\iint_{D\setminus (B_{\rho}(\zeta_1)\cup B_{\rho}(\zeta_2)} \left\vert \frac{1}{(z-\zeta_1)(z-\zeta_2)} \right\vert dx\wedge dy &\leq \frac{1}{\rho^2} \iint_{D\setminus (B_{\rho}(\zeta_1)\cup B_{\rho}(\zeta_2)} dx \wedge dy \nonumber \\
&\leq \frac{1}{\rho^2}\iint_{D} dx\wedge dy \leq \frac{Area(D)}{\rho^2}. \label{6.1}
\end{align}
And,
\begin{align}
\iint_{B_{\rho}(\zeta_i)} \left\vert \frac{1}{(z-\zeta_i)(z-\zeta_j)} \right\vert dx\wedge dy &\leq \iint_{B_{\rho}(\zeta_i)} \left\vert \frac{1}{\rho(z-\zeta_i)} \right\vert dx\wedge dy \nonumber\\
&= \frac{1}{\rho}\iint_{B_{\rho}(\zeta_i)} \left\vert \frac{1}{(z-\zeta_i)} \right\vert dx\wedge dy \leq 2\pi \label{6.2} 
\end{align}

From \ref{6.1} and \ref{6.2}, $ \frac{1}{(z - \zeta_1)(z- \zeta_2)}$ is integrable.  Thus, by Stoke's theorem:

\begin{align*}
\iint_D  \frac{\partial \psi^i_k}{\partial \overline{z}} \frac{1}{(z - \zeta_1)(z- \zeta_2)} dx\wedge dy &= \iint_D d\left( \frac{\psi^i_k dz}{(z - \zeta_1) (z - \zeta _2)} \right)= \int_{\partial D} \frac{\psi^i_k dz}{(z-\zeta_1)(z-\zeta_2)}.
\end{align*}
Fix $\zeta_1$ and choose $\zeta_2$ such that $\vert \zeta_1 - \zeta_2 \vert< d(\zeta_1,\partial D))/2$.  Thus,
\begin{align*}
d(\zeta_1,\partial D) \leq d(\zeta_1,\zeta_2) + d(\zeta_2,\partial D) \leq \frac{d(\zeta_1,\partial D)}{2} + d(\zeta_2,\partial D).
\end{align*}
So, we have 
\[
d(\zeta_2,\partial D) \geq \frac{d(\zeta_1,\partial D)}{2}.  
\]
Hence,
\begin{align*}
\left\vert \int_{\partial D} \frac{\psi^i_k dz}{(z-\zeta_1)(z-\zeta_2)} \right\vert \leq \int_{\partial D} \frac{\vert\psi^i_k\vert \vert dz \vert}{\vert z-\zeta_1\vert \vert z-\zeta_2\vert} \leq \frac{4M_2}{d(\zeta_1, \partial D)^2} \int_{\partial D} \vert dz \vert = \frac{8\pi R M_2}{d(\zeta_1,\partial D)}. 
\end{align*}
Thus, we get equicontinuity at any point $\zeta_1 \in D$.
\end{proof}
\noindent For simplicity of notation, we will denote the subsequence $\psi^i_{n_k} $ converging to $\psi^i $ by $\psi^i_k $.
\begin{lemma}
The map $\psi^i$ is holomorphic.
\end{lemma}

\begin{proof}From the definition of $\psi^i$, we have,
\begin{align*}
\psi^i(\zeta) &= \lim_{k\to \infty} \psi^i_k(\zeta)\\
&=  \frac{1}{2\pi \iota} \lim_{k\to \infty} \int_{\partial D} \frac{\psi^i_k(z) dz}{z - \zeta} - \frac{1}{\pi} \lim_{k\to \infty} \iint_D \frac{\partial \psi^i_k}{\partial \overline{z}} \frac{dx\wedge dy}{z - \zeta}
\end{align*}
By the dominated convergence theorem, we can interchange limit and integral.  Notice that,   
\begin{align*}
\lim_{k\to \infty} \iint_D \frac{\partial \psi^i_k}{\partial \overline{z}} \frac{dx\wedge dy}{z - \zeta} = \iint_D \lim_{k\to \infty} \frac{\partial \psi^i_k}{\partial \overline{z}} \frac{dx\wedge dy}{z - \zeta}
\end{align*}
Further, $\frac{\partial \psi^i_k}{\partial \overline{z}}$ converges to zero as $k$ tends to infinity, by Lemma \ref{approx_holo_disk}.  Thus,
\begin{align*}
\psi^i(\zeta) = \frac{1}{2\pi \iota} \lim_{k\to \infty} \int_{\partial D} \frac{\psi^i_k(z) dz}{z - \zeta} = \frac{1}{2\pi \iota} \int_{\partial D} \lim_{k\to \infty} \frac{\psi^i_k(z) dz}{z - \zeta} = \frac{1}{2\pi \iota} \int_{\partial D} \frac{\psi^i(z) dz}{z - \zeta}.
\end{align*} 
So, $\psi^i$ satifies the Cauchy integral formula.  Hence it is holomorphic. 

\end{proof}

\section{Step 5: Construction of a non-trivial lamination}
Pick a point $x$ and a sequence of points $x_k \in s_k^{-1}(0)$  such that the sequence converges to $x$.  Construct 
approximately holomorphic disks $\varphi_k$ centred at $x_k$ as in Step 2.  We saw in 
the previous section that $\psi^i_k = p_i\circ \varphi_k$ has a convergent subsequence, 
say $\psi_{n_k}^i$.  That is the functions $\varphi_{n_k}$ converge to an analytic disk $
\varphi: D_r \to \mathbb{CP}^2$ at $x$.  It is important to note that the radius $r$ of 
the disk is independent of the point $x$.  For simplicity of notation we will denote 
$\varphi_{n_k}$ by $\varphi_k$.

Fix $0 < \epsilon < r $. Let $x' = \varphi(r - \varepsilon)$. The sequence $x'_k = \varphi_k(r - \varepsilon)$ 
converges to $x'$.  Construct approximately holomorphic disks (of radius $r$) $\varphi'_k$ centred at 
$x'_k$ as in Step 2.  There exists a further subsequence, say $n_k$, such that 
$\varphi'_{n_k}$ converges to an analytic disc $\varphi':D_r \to \cp^2$ at $x'$.  This
disk overlaps the previous disk in an open set.  We have thus continued the disk
``forward''.  We construct the maximal surface which is obtained by continuing 
this 
way.  Let us call this surface $S$.  In order to prove that $S$ is an embedded 
surface, we need the following lemma.

\begin{lemma}\label{totally_geodesic_disk}
Let $\psi: D_r \to \cp^2$ be an analytic disk obtained as above in the limit. Then 
$ \psi (D_r) $ is a totally geodesic submanifold of $\cp^2 $ (with respect to the 
Fubini-Study metric). Therefore, the curvature of $\psi (D_r) $ at any point 
equals $1$.
\end{lemma}
\begin{proof}
By Lemma \ref{pointwise_zero} and expression (\ref{rel_connection_1}), we 
conclude that the Levi-Civita connection $\cdr $ and its restriction $
\overline{\cdr} $ to $\psi(D_r) $ are equal. Therefore, $\psi (D_r) $ is a totally 
geodesic submanifold of $\cp^2 $.
\end{proof}

\begin{lemma}
Let two analytic disks $\psi: D_r \to \cp^2$ and $\varphi: D_r \to \cp^2$ , 
obtained as limits, intersect nontrivially. Then the disks $\psi (D_r) $ and $
\varphi (D_r)$ intersect in an open subset of both or they intersect 
transversally.
\label{transoropen}
\end{lemma}

\begin{proof}
By Lemma \ref{totally_geodesic_disk}, the embedded disks $\psi (D_r) $ and $
\varphi (D_r)$ are totally geodesic submanifolds of (real) codimension $2$. 
Further, the curvature at point on the disks equals $1$.  We know that the 
sectional curvature $1$ is attained only at complex subspaces of the tangent space  
at any point in $\cp^2 $. 

Now, we recall a general result in this context that given a point $p$ in a 
Riemannian manifold $(M,g) $ and a subspace $ V$ of $T_p M $, if there is a 
totally geodesic submanifold in a neighborhood of $p$ which passes through 
$p$ and tangent to $ V$ then it must be unique. We apply it our setting to 
conclude that any given point $p$ and a 
complex subspace $\xi$ of $T_p \cp^2 $, there exists a unique totally geodesic 
submanifold passsing through $p$ and tangent to $\xi $ in a neighborhood of point 
$p$.  

Using the uniqueness result stated above, we conclude that either $\psi (D_r) $ 
and $\varphi (D_r)$ intersect transversally or they overlap in an open subset of 
both.
\end{proof}

\begin{lemma}
The surface $S$ is holomorphically embedded in $\cp^2$. 
\label{surfaceembedded}
\end{lemma}

\begin{proof} First we show that $S$, obtained as above, is embedded.
Assume the contrary. By Lemma \ref{transoropen}, this is possible if and only if a disk produced by taking the limit is intersected transversally by a disk formed later by the same process.  
Note that, at each step we took a further subsequence, so there exist a subsequence, 
$n_k$, such that both these disks are limits of disks in $s_{n_k}^{-1}(0)$.  But, this 
would imply that $s_{n_k}^{-1}(0)$ would intersect iteself for large $k$, which is 
impossible.
\end{proof}

\begin{theorem}\label{lamination}
If $S$ is a surface obtained as above, then $\overline{S}$ is laminated.  
\end{theorem}

\begin{proof}
Let $x\in \overline{S}$.  Then there exists a sequence $x_n\in S$ such that $x_n$ 
converges to $x$.  Further, there exist disks $\varphi_n:D_r \to \mathbb{CP}^2$ such 
that $\varphi_n(0)=x_n$ and $\varphi_n(D_r)\subset S$.  As before, there exists some 
subsequence of these $\varphi_n$'s that converge, say $\varphi_{n_k}$ converges to $
\varphi: D_r \to \mathbb{CP}^2$.  This gives us a disk at $x$.  

In addition, we observe that the following holds.
\begin{lemma}
The sequence $\varphi_n$ converges to $\varphi$.
\end{lemma}
\begin{proof}
Assume contrary that there exists a subsequence $\varphi_{l_k}$ that converges to 
$\varphi': D_r \to \mathbb{CP}^2$.  Notice that $\varphi$ and $\varphi'$ has to 
overlap. Otherwise, the images of the maps $\varphi_{l_k}$ and $\varphi_{n_k}$ 
will intersect for large $k$, which will 
imply intersection of $S$ with itself, which will contradict Lemma \ref{surfaceembedded}.
\end{proof}

\noindent Further, by the same idea, we can see that,

\begin{lemma}
The disk, $D_x$ at $x$, as obtained above, does not depend on the choice of the sequence $x_n$.
\end{lemma}

\noindent Thus, construct a disk at all points in $\overline{S}$ in this manner.  
In the following lemma, we prove that this intersection cannot be transverse.

\begin{lemma}
The disk $D_x$ cannot intersect the disk $D_y$ at $y$ transversally.
\label{disktrans}
\end{lemma}

\begin{proof}
Assume the contrary.  Let $x_n\in S$ and $y_n\in S$ be points such that 
$x_n$ converge to $x$ and $y_n$ converge to $y$.  Let $D_{x_n}$ be disks at $x_n$ 
contained in $S$ and $D_{y_n}$ be disks at $y_n$ contained in $S$. The disk $D_x$ 
intersects $D_y$ transversally implies that $D_{x_n}$ intersects $D_{y_n}$ 
transversally for all 
$n$ sufficiently large.  This implies that $S$ intersects itself, contradicting 
Lemma \ref{surfaceembedded}. 
\end{proof}

Lemma \ref{transoropen} and Lemma \ref{disktrans} ensures that all these disks will have to intersect in an open set. So, we have a leaf through every point in $\overline{S}$.  Thus, $\overline{S}$ is laminated.  
\end{proof}

\begin{proof}[Proof of Theorem \ref{main_thm}]

If the surface $S$ given above is not compact, then $\overline{S}$ is a non-trivial 
lamination as desired.  So, we assume the contrary, that is, $S$ is a compact surface. 

\begin{lemma}
If $S$ is compact, there exists a subsequence $s_{n_k} $ such that every point in $S$ is a limit of points in $s_{n_k}^{-1}(0) $.
\end{lemma}

\begin{proof}
As, $S$ is compact, it will be covered by finitely many holomorphic disks.  So, the process of continuing the disk forward will stop after finite steps.  So, the subsequence, $n_k$ for which the final disk is a limit of disks in $s_{n_k}^{-1}(0)$, will work.
\end{proof}

Rename this sequence as $s_k$.  We recall Proposition 40 in \cite{DS} which states that the limit of currents induced 
by $W_{k} $'s has $\cp^2 $ as its support. So, there exists a point $y \in S^C$ and a 
sequence $y_{k}\in s_{k}^{-1}(0)$ such that $y_{k}$ converges to $y$. Construct a maximal 
surface as before passing through $y$. We call it $R$.  If $R$ is compact, then $R$ must intersect $S$ by Bezout's theorem.  This will imply that $s_{n_k}^{-1}(0)$ will 
intersect itself for sufficiently large $k$. So, $R$ cannot be compact. Therefore, $\overline{R}$ is a non-trivial 
lamination, by Theorem \ref{lamination}. 

Now, to show that each leaf of the lamination, obtained as above, is a totally geodesic 
submanifold, we recall Lemma \ref{pointwise_zero} which says that Weingarten 
maps go to zero pointwise. This implies that the induced connection $\overline{\cdr} $ on each leaf agrees with the Levi-Civita connection $\cdr $ on $\cp^2 $ from relation \ref{rel_connection}.
\end{proof}

\bibliographystyle{alpha}
\bibliography{ExistenceOfNonTrivialLaminations}

\end{document}